\documentclass[12pt]{amsart}

\usepackage{amsthm}
\usepackage{amsmath,amssymb,amsfonts,marvosym,amscd}
\usepackage{tikz}
\usepackage{tikz-cd}
\usepackage{wasysym} 

\usepackage{pifont} 
\usepackage{graphicx}
\usepackage{psfrag}
\usepackage{datetime}
\usepackage{enumerate}

\usepackage{geometry} 
\usepackage[all]{xy}

\usepackage{hyperref}

\title{Fabric idempotents and homological dimensions}
\author{Jordan McMahon}
\begin{document}

\newtheorem{lm}{Lemma}[section]
\newtheorem{prop}[lm]{Proposition}
\newtheorem{conj}[lm]{Conjecture}
\newtheorem{cor}[lm]{Corollary}
\newtheorem{theorem}[lm]{Theorem}

\theoremstyle{definition}
\newtheorem{eg}{Example}
\newtheorem{defn}{Definition}
\newtheorem{remark}{Remark}
\newtheorem{qu}{Question}

\begin{abstract}
Over a finite-dimensonal algbera $A$, simple $A$-modules that have projective dimension one have special properties. For example, Geigle-Lenzing studied them in connection to homological epimorphisms of rings, and they have also appeared in work concerning the finitistic dimension conjecture. If we however work in a $d$-cluster-tilting subcategory, then not all simples are contained in this subcategory. In this context, a replacement might be to work with idempotent ideals instead, and utilise the theory of Auslander-Platzeck-Todorov. We introduce the notion of a fabric idempotent as an analogue of the localising modules studied by Chen-Krause, and to illustrate the theory we show that they provide rich combinatorial properties. An application is to extend the classification of singularity categories of Nakayama algebras by Chen-Ye to higher Nakayama algebras.
\end{abstract}

\maketitle
\tableofcontents

\newcommand{\sub}{\underline}
\newcommand{\ra}{\rightarrow}
\newcommand{\mc}{\mathcal}
\newcommand{\blp}{^\prime}
\newcommand{\mcm}{\operatorname{\sub{MCM}(A)}}
\newcommand{\ba}{\begin{enumerate}[(a)]}

\section{Introduction}

For a given algebra, the study of its idempotents often contains useful information about the structure of the algebra itself. In the study of finite-dimensional algebras over a field, this is no different. The work of Auslander-Platzeck-Todorov \cite{apt} forms a starting point for connecting properties of idempotent ideals and homological dimensions. 
%Another focus will be when idempotent ideals define localising subcategories, as studied in \cite{gl}. Localising subcategories are related to the shrinking of arrows procedure detailed by Ringel (cited in \cite{gl}). Further work in this direction can found in \cite{ck}, \cite{cy}, and we connect this with the notion of an eventually homological isomorphism, introduced in \cite{pss}. %reference here!!!!!!!!!!
%Tying the two concepts will be relative homology, see \cite{is}, \cite{as}, \cite{la}.
The benefit of studying idempotent ideals is that homological information is often easily transferred. \begin{defn}[Definition \ref{fabdef}]
	Let $A$ be a finite-dimensional algebra and $f$ an idempotent of $A$. Then $f$ is an \emph{fabric idempotent of $A$} if 
	\begin{itemize}
		\item The idempotent $f$ satisfies $\mathrm{proj.dim}_A(A/\langle f \rangle) \leq 1$, and there exists an idempotent $e$ such that:
		\item For every projective $A/\langle f \rangle$-module $P$, then $\tau_A(P)$ is injective as an \\$A/\langle e \rangle$-module.
\item For every injective $A/\langle e \rangle$-module $I$, then $\tau_A^{-1}(I)$ is projective as an \\$A/\langle f \rangle$-module.			\end{itemize}
\end{defn}
In the case that $A/\langle f\rangle=:S$ is simple as an $A$-module, then it is a localisable object: the module $S$ is simple and satisfies $\mathrm{proj.dim}_A(S)\leq 1$ and $\mathrm{Ext}^1_A(S,S)=0$. Localisable objects were introduced in \cite{ck}: the name stems from the localising subcategories studied by Geigle and Lenzing in \cite[Section 2]{gl}. Roughly speaking, localising modules allow for the conservation of information upon contraction of a arrow in a path algebra. Contractions of subquivers of type $A$ in a path algebra were considered in \cite{hkp}. The weaker condition of a simple module having projective dimension one was also of interest for the contracting of arrows as a reduction technique for finitistic dimension was studied in \cite{gps}. This was inspired by a reduction technique of Fuller-Saorin \cite{fs}, see the work of Xi \cite{xi} that deals with idempotents and finitistic dimension. 

An initial motivation for the introduction of localising modules was to study weighted projective lines as in \cite{ck2}, \cite{ck}, in the same vein as Geigle-Lenzing \cite{gl}. Another application was in the description of singularity categories of Nakayama algebras by Chen-Ye \cite{cy}.  For a finite-dimensional algebra with fabric idempotent $f$, suppose that an indecomposable $A$-module $P$ that is projective as an $A/\langle f\rangle$-module can be expressed in the form $P=\Omega^n_A(I)$ for some injective $A$-module $I$ and some postive integer $n$, and such that $n$ is minimal with this property. Then we say the \emph{fabric dimension} of $P$ is $\mathrm{fab.dim}_A(P)=n$, and say $\mathrm{fab.dim}_A(P)=\infty$ if $P$ cannot be expressed in this fashion. We further define $$\mathrm{fab.dim}_A(A/\langle f\rangle):=\mathrm{sup}\{\mathrm{fab.dim}(P)|P\in\mathrm{ind}(\mathrm{proj}(A/\langle f \rangle))\}.$$ Fabric idempotents allow us to use a variety of results from the literature to give a precise description of which projective modules generate each term in the projective resolution of a given Gorenstein injective module. 

\begin{prop}[Proposition \ref{fabdd}]
	Let $A$ be an $n$-Iwanaga-Gorenstein algebra with fabric idempotent $f$ satisfying $\mathrm{fab.dim}_A(A/\langle f \rangle)=n-1$. Then the following holds\begin{align*} 
	DA\in \mathrm{gen}_{m}(Ah) \implies M&\in \mathrm{gen}_{m}(Ah)  , &&\text{ for all $M\in \mathrm{GI}(A)$, $m<n$};\\
	\Omega^n(M)&\in  \mathrm{gen}_{\infty}(Af),&&\text{ for all $M\in \mathrm{mod}(A)$}.
	\end{align*}  
\end{prop}

In other words, for an $n$-Iwanaga-Gorenstion algebra $A$ with fabric idempotent $f$, we may consider any Gorenstein injective $A$-module $M$. If we track the terms in the projective resolution of $M$, there will be an integer $n$, determined by the fabric dimension, such that the terms in the resolution switch to having a different projective generator after $n$ terms. Fabric idempotent ideals are quite natural, and we give multiple examples of algebras that admit a fabric idempotent ideal. In particular, we give examples of higher canonical algebras, introduced in \cite{himo}, minimally Auslander-Gorenstein algebras, introduced in \cite{is}, as well as higher Nakayama algebras, as introduced in \cite{jk}. By ``contracting" along fabric idempotents, generalising the methods of Chen-Ye \cite{cy} on Nakayama algebras, we have the following result.

\begin{theorem}[Theorem \ref{nakay}]
Let $A$ be a finite-dimensional higher Nakayama algebra. Then there exists a self-injective, higher Nakayama algebra $B$ such that 
$$D_{\mathrm{sg}}(A)\cong D_{\mathrm{sg}}(B)\cong \underline{\mathrm{mod}}(B).$$
\end{theorem}

 The class of Gorenstein projective modules for Nakayama algebras have also been studied in \cite{ringnak} and singularity categories of Nakayama algebras have been further studied in \cite{shen}.

\section{Background and notation}
Let $Q$ be a quiver and 
Let $A$ be a finite-dimensional algebra over a field $k$: we will assume that $A$ is of the form $kQ/I$, the path algebra $kQ$ over some quiver $Q$ and $I$ is an admissible ideal of $kQ$. For two arrows in $Q$ $\alpha:i\rightarrow j$ and $\beta:j\rightarrow k$, we denote their composition as $\beta\alpha:i\rightarrow k$. Let $A^\mathrm{op}$ denote the opposite algebra of $A$. An $A$-module will mean a finitely-generated left $A$-module; by $\mathrm{mod}(A)$ we denote the category of finite-dimensional left $A$-modules. The functor $D=\mathrm{Hom}_k(-,k)$ defines a duality between $\mathrm{mod}(A)$ and $\mathrm{mod}(A^{\mathrm{op}})$. Let $\nu_A=D\mathrm{Hom}(-,A)$ denote the Nakayama functor, $\tau_A$ the Auslander-Reiten translation and $\Omega_A$ the szyzgy functor. We omit the subscripts when the context is clear.  For an $A$-module $M$, let $\mathrm{add}(M)$ be the full subcategory of $\mathrm{mod}(A)$ composed of all $A$-modules isomorphic to direct summands of finite direct sums of copies of $M$. We further denote by $\mathrm{proj}(A)$ the class of projective $A$-modules, $\mathrm{inj}(A)$ the class of injective $A$-modules. For a class of modules $\mathcal{C}\subseteq \mathrm{mod}(A)$, we set $\mathrm{ind}(\mathcal{C})$ to be the indecomposable $A$-modules that are in $\mathcal{C}$. 

 An element $e$ of $A$ is an \emph{idempotent} if $e^2=e$. Two idempotents $e_1$ and $e_2$ are \emph{orthogonal} if $e_1e_2=e_2e_1=0$; an idempotent $e$ is \emph{primitive} if $e$ cannot be written as a sum $e=e_1+e_2$ for two non-zero, orthogonal idempotents  $e_1$ and $e_2$. For a finite-dimensional algebra $A$, the module $_AA$ admits a direct sum decomposition $$_AA=Ae_1\oplus \cdots \oplus Ae_n,$$ where $1=e_1+ \cdots + e_n$ and $e_1,\ldots,e_n$ are primitive, pairwise orthogonal, non-zero idempotents. The modules $Ae_i$ are the indecomposable projective $A$-modules.

Let $M=M_1\oplus \cdots  \oplus M_t$ be a basic $A$-module, in that the $M_i$ are pairwise non-isomorphic, indecomposable $A$-modules for $1\leq i\leq t$. Then for each $i$, the module  $\mathrm{Hom}_A(M,M_i)$ is an indecomposable, projective $\mathrm{End}_A(M)^{\mathrm{op}}$-module (and every indecomposable, projective $\mathrm{End}_A(M)^{\mathrm{op}}$-module is of this form). In this sense, each summand $M_i$ of $M$ corresponds to a primitive idempotent of the algebra $\mathrm{End}_A(M)^{\mathrm{op}}$. 

An algebra $A$ is called an \emph{$n$-Iwanaga-Gorenstein algebra} if it satisfies the following axioms:

\begin{enumerate} 
\item $\mathrm{inj.dim}_A(A)\leq n$,
\item $\mathrm{proj.dim}_A(DA)\leq n.$
\end{enumerate}

For example, any finite-dimensional algebra with finite global dimension $n$, such as the path algebra of a quiver without oriented cycles, is trivially $n$-Iwanaga-Gorenstein. An $n$-Iwanaga-Gorenstein algebra is trivially also $m$-Iwanaga-Gorenstein for all $m>n$; we denote by $\mathrm{Gor.dim}(A)$ the minimal positive integer $n$ such that $A$ is $n$-Iwanaga-Gorenstein. An $A$-module $M$ is said to be \emph{Gorenstein projective} (also referred to as \emph{maximal Cohen-Macaulay} in the literature, most notably in \cite{buch}) if $\mathrm{Ext}^i_A(M,A)=0$ for all $i>0$. The class of Gorenstein projective modules is denoted $\mathrm{GP}(A)$, and the class of Gorenstein projective modules that are not themselves projective is denoted by $\underline{\mathrm{GP}}(A)$. Likewise, define a module $M$ to be \emph{Gorenstein injective} if $\mathrm{Ext}^i_A(DA,M)=0$ for all $i>0$, and denote by $\mathrm{GI}(A)$ the class of Gorenstein injective modules and $\underline{\mathrm{GI}}(A)$ the class of Gorenstein injective modules that are not injective.

Let $D^b(A)$ denote the bounded derived category of $\mathrm{mod}(A)$. A complex of $A$-modules is said to be {\em perfect} if it is isomorphic in $D^b(A)$ to a finite complex of finitely generated projective $A$-modules. This gives a full subcategory of $D^b(A)$, denoted by $D_{\mathrm{perf}}^b(A)$. The \emph{singularity category} $D_\mathrm{sg}(A)$ is defined as the Verdier quotient of $D^b(A)/D^b_{\mathrm{perf}}(A)$ \cite{buch} \cite{orlov}. The following theorem is a classical result of Buchweitz.

\begin{theorem}\cite[Theorem 4.4.1]{buch}
Let $A$ be an $n$-Iwanaga-Gorenstein algebra. Then there is an equivalence of categories:
$$\underline{\mathrm{GP}}(A)\cong D_{\mathrm{sg}}(A).$$
\end{theorem}
To determine the Gorenstein projective modules over $A$, the following result is useful. 

\begin{prop}\label{omega}\cite[Theorem 2.3.3]{chen}
For a finite-dimensional algebra $A$, the following are equivalent:
\begin{itemize}
\item The algebra $A$ is $n$-Iwanaga-Gorenstein.
\item There is an equality $\mathrm{GP}(A)=\Omega^n(\mathrm{mod}(A))$,
\end{itemize}
where $\Omega^n(\mathrm{mod}(A))=\{\Omega^n(M)|M\in\mathrm{mod}(A)\}$, and includes the projective $A$-modules. 
\end{prop}

\section{Fabric idempotent ideals}

Our reference point for the study of idempotent ideals will be the work of Auslander-Platzeck-Todorov \cite{apt}. An \emph{idempotent ideal of $A$} is a two-sided ideal $I$ of $A$ such that $I^2=I$. In our setting, every idempotent ideal will be of the form $AeA$ for some idempotent $e$ of $A$, and we write $\langle e \rangle:=AeA$ for this ideal.

\begin{defn}
For an integer $l\geq 0$ and idempotent $e$ of $A$, define the class of modules $\mathrm{gen}_{l}(Ae)$ by $M\in \mathrm{gen}_{l}(Ae)$ if and only if $M$ has a projective resolution $$\cdots \ra P_1\ra P_0\ra M \ra 0$$ such that $P_i\in \mathrm{add}(Ae)$ for all $0\leq i\leq l$. If for every $i\geq 0$, $P_i$ is contained in $\mathrm{add}(Ae)$, we say $M\in \mathrm{gen}_{\infty}(Ae)$.
Dually, define the class of modules $\mathrm{cogen}_{l}(eA)$ by $M\in \mathrm{cogen}_{l}(eA)$ if and only if $M$ has a injective resolution $$0 \ra M\ra I_0\ra I_1 \ra \cdots$$ such that $I_i\in \mathrm{add}(eA)$ for all $0\leq i\leq l$. The class $\mathrm{cogen}_{\infty}(eA)$ may be defined similarly. 
\end{defn}

Here we have used the notation of \cite{ps}. The notation of Auslander-Platzeck-Todorov \cite[Defintion 2.3]{apt} assumes a fixed projective module $P$. If for example $P=Ae$ for some idempotent $e$ then the class $\mathrm{gen}_l(Ae)$ is denoted in \cite{apt} by $\mathbf{P}_l$.

\begin{prop}\cite[Proposition 2.4]{apt}\label{tooapt}
For an arbitrary $A$-module $M$, then $$M\in\mathrm{cogen}_{l}(eA) \iff \mathrm{Ext}^i_A(A/\langle e \rangle,M)=0$$ for all $0\leq i\leq l$.
Similarly, $$M\in\mathrm{gen}_{l}(Ae) \iff \mathrm{Ext}^i_A(M,F)=0$$ for all injective $A/\langle e \rangle$-modules $F$ and $0\leq i\leq l$.
\end{prop}
From this result we may obtain a great deal of information, especially in the case of fabric idempotent ideals. In order to understand the existence of fabric idempotents, we may restrict our attention to algebras satisfying the following conditions on their quivers. 

\begin{prop}\label{masod}
Let $A=kQ/I$ be a finite-dimensional algebra and let $F$ be a subset of the set of vertices of $Q$ and $f=\sum_{i\in F}e_i$. Suppose that $$\mathrm{proj.dim}_A(A/\langle f \rangle)\leq 1.$$ Then there exists a unique idempotent $e$ of $A$ such that $f$ becomes a fabric idempotent if and only if the following conditions hold:
\begin{enumerate}
\item For each $i\notin F$ there is at most one arrow $\alpha_i:i\rightarrow t(\alpha_i)$ starting at $i$ and such that $t(\alpha_i)\in F$. If such an arrow exists, define $t(\alpha_i)=:i^\prime$. \label{ef1}
\item For any $j\notin F$ and arrow $\delta:s(\delta)\rightarrow j^\prime$, then either $s(\delta)=j$ or $s(\delta)=i^\prime $ for some $i\notin F$.   \label{ef4}
\item Let $ j\notin F$ such that an arrow $\alpha_j$ exists. Then for any $i\notin F$ and any arrow $\beta:i\rightarrow j$, then either $\alpha_j\beta\in I$ or there exist arrows $\alpha_i$ and $\delta:i^\prime \rightarrow j^\prime$ such that $\delta\alpha_i\cong \alpha_j\beta$.\label{ef2}
\item For any $i,j\notin F$ and arrow $\delta:i^\prime \rightarrow j^\prime$, then either $\delta\alpha_i\in I$ or there exists an arrow $\beta:i\rightarrow j$ such that $\delta\alpha_i \cong \alpha_j\beta$. \label{ef3}
\end{enumerate}

\end{prop}

These conditions should be seen as a generalisation of the conditions described in \cite[Example 3.6.2]{ck}. In particular, if $A$ is hereditary and both $A/\langle e\rangle$ and $A/\langle f\rangle$ are simple as $A$-modules, then they determine a non-split expansion of abelian categories.

\begin{proof}
Set $E$ to be the subset of vertices that do not arise as $i^\prime$ for some $i\notin F$, and  let $e=\sum_{i\in E}e_i$.
If $i\notin F$ and there is no such arrow $\alpha_i$, then condition (\ref{ef2}) implies that $Ae_i/\langle f \rangle $ is projective as a $A$-module. In this case $\tau_A(Ae_i/\langle f\rangle)=0$. Otherwise, condition (\ref{ef1}) implies that for $i\notin F$ there is a sequence \begin{align*}
 A e_{i^\prime}\rightarrow A e_{i}\rightarrow A e_i/\langle f \rangle\rightarrow 0.
\end{align*}
This sequence is exact if and only if it is true for each $j\in F$ that any path $p: i\rightarrow \cdots \rightarrow j$ factors through $\alpha_i$. Hence condition (\ref{ef2}) implies exactness.
Because we assume that $\mathrm{proj.dim}_A(A/\langle f \rangle)\leq 1$, then there must be a short exact sequence 

\begin{align}\label{fe2}
0\rightarrow A e_{i^\prime}\rightarrow A e_{i}\rightarrow A e_i/\langle f \rangle\rightarrow 0.
\end{align}  Dually, conditions (\ref{ef4}) and (\ref{ef3}) imply that there is an exact sequence\begin{align}\label{fe1}
 0\rightarrow e_{i^\prime}A/\langle e \rangle\rightarrow e_{i^\prime}A\rightarrow e_{i}A.
\end{align}
Therefore $\tau_A(A/\langle f \rangle )=DA/\langle e\rangle$.
Conversely, let $\tau_A(A/\langle f \rangle )=DA/\langle e\rangle$. The Auslander-Reiten translation $\tau$ sends indecomposable non-projective modules to indecomposable non-injective modules. So let $Ae_i/\langle f \rangle$ be an indecomposable, non-projective $A$-module; it follows that there is a vertex $i^\prime\in Q$ such that the module $\tau(Ae_i/\langle f \rangle)=e_{i^\prime}A/\langle e \rangle$ is an indecomposable, non-injective $A$-module. By virtue of the fact that the projective cover of $Ae_i/\langle f \rangle$ is $Ae_i$ and the injective envelope of $e_{i^\prime}A/\langle e \rangle$ is $e_{i^\prime}A$, we must have exact sequences:
\begin{align}\label{fe3}
&0\rightarrow A e_{i^\prime}\rightarrow A e_{i}\rightarrow A e_i/\langle f \rangle\rightarrow 0.\\
& 0\rightarrow e_{i^\prime}A/\langle e \rangle\rightarrow e_{i^\prime}A\rightarrow e_{i}A.\label{fe4}
\end{align}
So there must be an arrow $\alpha_i:i\rightarrow i^\prime$, and this is the unique arrow starting at $i$ and ending at a vertex in $F$, likewise the unique arrow ending at $i^\prime$ and starting at a vertex in $E$.  This implies that conditions (\ref{ef1}) and (\ref{ef4}) hold. The exactness of the equations \eqref{fe3} and \eqref{fe4} now implies that the remaining conditions hold.

If $Ae_i/\langle f\rangle$ is an indecomposable, projective $A$-module, then there may be no arrow as in condition (\ref{ef1}), and the remaining conditions hold trivially.
\end{proof}

An implication of Proposition \ref{masod} is that the definition of a fabric idempotent is equivalent to the following:

\begin{defn}\label{fabdef}
Let $A=kQ/I$ be a finite-dimensional algebra and let $F$ be a subset of the set of vertices of $Q$ and $f=\sum_{i\in F}e_i$. Then $f$ is an \emph{fabric idempotent of $A$} if 
\begin{itemize}
	\item The idempotent $f$ satisfies $\mathrm{proj.dim}_A(A/\langle f \rangle) \leq 1$.
	\item Conditions (\ref{ef1})-(\ref{ef3}) of Proposition \ref{masod} are satisfied. 
\end{itemize}
Dually, $f$ is a \emph{cofabric idempotent of $A$} if 
\begin{itemize}
	\item The idempotent $f$ satisfies $\mathrm{inj.dim}_A(I) \leq 1$ for any injective $A/\langle f\rangle$-module $I$, and there exists an idempotent $e$ such that:
	\item For any injective $A/\langle f\rangle$-module $I$, then $\tau^{-1}_A(I)$ is projective as an\\ $A\langle e \rangle$-module.
	\item For any projective $A/\langle f\rangle$-module $P$, then $\tau_A(P)$ is injective as an \\$A\langle e \rangle$-module.
\end{itemize} The \emph{cofabric dimension} can be defined in a dual fashion to the fabric dimension: $$\mathrm{cofab.dim}_A(A/\langle f\rangle):=\mathrm{sup}\{\mathrm{cofab.dim}(I)|I\in\mathrm{ind}(\mathrm{inj}(A/\langle f \rangle))\},$$ where $\mathrm{cofab.dim}(I)$ is the minimal positive integer $n$ (if one exists) such that $I=\Omega^{-n}_A(P)$ for some projective $A$-module $P$ and infinite if no such projective module exists.
\end{defn}
The following result  restricts to Lemma 3.9 of \cite{mz} if $Ae$ is an additive generator of the  projective-injective $A$-modules. 
\begin{lm}\label{mar} Let $A$ be an $n$-Iwanaga-Gorenstein algebra and $DA\in\mathrm{gen}_{n-1}(Ae)$. Then $Q\in\mathrm{gen}_{n-1}(Ae)$ for all $Q\in \overline{\mathrm{GI}}(A)$. 
\end{lm}

\begin{proof}
By Proposition \ref{omega}, any $Q\in \overline{\mathrm{GI}}(A)$ is of the form $\Omega^{-n}(N)$ for some $A$-module $N$. So there is an exact sequence 
$$0\rightarrow N \rightarrow I_0 \rightarrow \cdots \rightarrow I_{n-1} \rightarrow Q\rightarrow 0.$$
This induces exact an exact sequence for any injective $A/\langle e \rangle$-module $F$.
$$0\rightarrow \mathrm{Hom}_A (\Omega^{-1}(N),F)\rightarrow \mathrm{Hom}_A(I_{0},F).$$ Since $DA\in\mathrm{gen}(Ae)$, Lemma \ref{tooapt} implies that $\mathrm{Hom}_A(I_0,F)=0$. This shows that $ \mathrm{Hom}_A (\Omega^{-i}(N),F)=0$ for all $1\leq i \leq n$. There are exact sequences for all $1\leq i\leq n-2$
$$ \mathrm{Ext}^{i-1}_A (\Omega^{-i}(N),F)\rightarrow \mathrm{Ext}^{i}_A (\Omega^{-i-1}(N),F)\rightarrow \mathrm{Ext}^{i}_A(I_{i},F).$$ Hence, by induction, $\mathrm{Ext}^i_A(Q,F)=0$ for all $0\leq i\leq n-1$, and any injective $A /\langle e \rangle$-module $F$. Therefore, by Proposition \ref{tooapt}, $Q\in\mathrm{gen}_{n-1}(Ae)$ for all $Q\in \overline{\mathrm{GI}}(A)$. 
 \end{proof}

Another simple calculation shows the following:
\begin{lm}\label{desc}
Let $A$ be an $n$-Iwanaga-Gorenstein algebra with fabric idempotent $f$ such that $\mathrm{fab.dim}_A(A/\langle f \rangle)<n$. Then for any $M\in\mathrm{mod}(A)$, we have $$\Omega^n(M)\in \mathrm{gen}_\infty (Af).$$  

\end{lm}

\begin{proof}
Let $Q=\Omega^{n}(M)$ for some $A$-module $M$. For any indecomposable projective $A/\langle f\rangle$-module $P$, let $\mathrm{fab.dim}(P)=m$. Then $\Omega^m(I)=P$ for some injective $A$-module $I$. We calculate for $i>0$,

\begin{align*}
\mathrm{Ext}_A^{i}(P,Q)&=\mathrm{Ext}_A^{i}(\Omega^{m}(I), \Omega^n(M))\\
&=\mathrm{Ext}^{i}_A(I, \Omega^{n-m}(M))\\
&=0,
\end{align*}
as $n>m$. So Proposition \ref{tooapt} implies that $Q\in \mathrm{cogen}_{\infty}(fA)$.
\end{proof}

\begin{prop}\label{fabdd}
	Let $A$ be an $n$-Iwanaga-Gorenstein algebra with fabric idempotent $f$ satisfying $\mathrm{fab.dim}_A(A/\langle f \rangle)=n-1$. Then the following holds\begin{align*} 
	DA\in \mathrm{gen}_{m}(Ah) \implies M&\in \mathrm{gen}_{m}(Ah)  , &&\text{ for all $M\in \mathrm{GI}(A)$, $m<n$};\\
	\Omega^n(M)&\in  \mathrm{gen}_{\infty}(Af),&&\text{ for all $M\in \mathrm{mod}(A)$}.
	\end{align*}  
Dually, if $A$ is an $n$-Iwanaga-Gorenstein algebra with cofabric idempotent $f$ satisfying $\mathrm{cofab.dim}_A(DA/\langle f \rangle)=n-1$. Then\begin{align*} 
	A\in \mathrm{cogen}_{m}(hA) \implies M&\in \mathrm{cogen}_{m}(hA)  , &&\text{ for all $M\in \mathrm{GP}(A)$, $m<n$};\\
	\Omega^{-n}(M)&\in  \mathrm{cogen}_{\infty}(fA),&&\text{ for all $M\in \mathrm{mod}(A)$}.
	\end{align*}  
\end{prop}
\begin{proof}
This is now Lemma \ref{mar}, Lemma \ref{desc}, and the dual results. \end{proof}

If $A$ has a fabric idempotent $f$, then it is possible to compare the singularity categories of $\mathrm{mod}(A)$ and of $\mathrm{mod}(fAf)$. 

\begin{theorem}\cite[Corollary 3.3]{chen3} \cite[Theorem 5.2]{pss}\label{chen}
Let $A$ be a finite-dimensional algebra and $f$ an idempotent of $A$. Then there is an equivalence $$D_{\mathrm{sg}}(A)\cong D_{\mathrm{sg}}(fAf)$$ 
if and only if the algebra $A$ satisfies $\mathrm{proj.dim}_A(M)<\infty$ for all modules \\ $M\in \mathrm{mod}(A/\langle f \rangle)$ and $\mathrm{proj.dim}_{fAf}(fA)<\infty$. 
\end{theorem}

This has the following consequence for fabric idempotetents.
\begin{cor}\label{chenii}
Let $A$ be a finite-dimensional algebra with fabric idempotent $f$. Then if $\mathrm{gl.dim}(A/\langle f\rangle)<\infty$, then there is an equivalence $$D_{\mathrm{sg}}(A)\cong D_{\mathrm{sg}}(fAf).$$  
\end{cor}

\begin{proof}
Apply $\mathrm{Hom}_A(Af,-)$ to the equation \eqref{fe2} in Proposition \ref{masod}. This shows that $fA$ is projective as an $fAf$-module.  Now, by definition $\mathrm{proj.dim}_A(A/\langle f\rangle )\leq 1$. So if $\mathrm{gl.dim}(A/\langle f\rangle)<\infty$, then by Lemma 5.8 of \cite{apt}, $\mathrm{proj.dim}_A(M)<\infty$ for all modules $M\in \mathrm{mod}(A/\langle f \rangle)$. 
Therefore the conditions of Theorem \ref{chen} are satisfied and there is a singular equivalence $$D_{\mathrm{sg}}(A)\cong D_{\mathrm{sg}}(fAf).$$  
\end{proof}

\section{Examples of fabric idempotent ideals}

\begin{eg}
Let $A=kQ/I$ be the algebra with quiver $Q$:

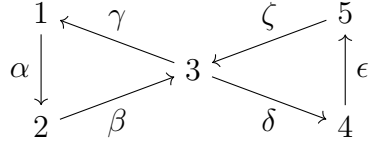
$$\begin{tikzpicture}
\node(d) at (-2,0.75){$1$};
\node(a) at (0,0){$3$};
\node(c) at (2,0.75){$5$};
\node(e) at (-2,-0.75){$2$};
\node(f) at (2,-0.75){$4$};

\path[->] (d) edgenode[left]{$\alpha$} (e);
\path[->] (e) edgenode[below]{$\beta$} (a);
\path[->] (a) edgenode[above]{$\gamma$} (d);

\path[->] (a) edgenode[below]{$\delta$} (f);
\path[->] (f) edgenode[right]{$\epsilon$} (c);
\path[->] (c) edgenode[above]{$\zeta$} (a);
\end{tikzpicture}$$
and relations $I=\langle\zeta\epsilon\delta-\beta\alpha\gamma,\alpha\gamma\beta,\epsilon\delta\zeta\rangle$. Let $f$ be the idempotent $e_2+e_3+e_5$, so that $_A A/\langle f \rangle=S_1\oplus S_4$.  We have the following short exact sequences:
\begin{align*}
0\rightarrow P_5\rightarrow P_4\rightarrow S_4\rightarrow 0;\\
0\rightarrow P_2\rightarrow P_1\rightarrow S_1\rightarrow 0.
\end{align*}

Hence $\mathrm{proj.dim}_A(A/\langle f\rangle))\leq 1$, and it can be checked (for example, using the conditions of Proposition \ref{masod}) that $$\tau_A(A/\langle f\rangle)=DA/\langle e \rangle$$ for the idempotent $e=e_1+e_3+e_4$. 
To check the third condition of Definition \ref{fabdef}, we use the resolutions,
\begin{align*}
0\rightarrow P_5\rightarrow P_4\rightarrow P_3\rightarrow I_1\rightarrow 0;\\
0\rightarrow P_2\rightarrow P_1\rightarrow P_3\rightarrow I_4\rightarrow 0.
\end{align*}

Hence, $S_4=\Omega_A(I_1)$ and $S_1=\Omega_A(I_4)$. Calculating the remaining projective resolutions

\begin{align*}
0\rightarrow P_5\rightarrow P_4\rightarrow P_4\rightarrow I_2\rightarrow 0;\\
0\rightarrow P_2\rightarrow P_1\rightarrow P_1\rightarrow I_5\rightarrow 0;\\
0\rightarrow P_2\oplus P_5\rightarrow P_1\oplus P_4\rightarrow P_3\rightarrow I_3\rightarrow 0,
\end{align*}
and we see that $DA\in \mathrm{gen}_1(Ae)$ and therefore $f$ is a fabric idempotent of $A$ satisfying $\mathrm{fab.dim}(A)=1$.

Applying Proposition \ref{fabdd} we know that for any $M\in \mathrm{GI}(A)$, the first two terms in the projective resolution of $M$ are generated by $Ae$ and the remaining terms are generated by $Af$. 
\end{eg}

\subsection{$n$-Canonical algebras}

Canonical algebras were introduced by Ringel \cite{ring}. The $n$-canonical algebras were introduced by Herschend-Iyama-Minamoto-Oppermann \cite{himo} to extend the class of canonical algebras to higher Auslander-Reiten theory. An explicit description of $n$-canonical algebras in terms of a quiver with relations was given in \cite[Theorem 3.46]{himo}. 

\begin{eg}
As an example, the $2$-canonical algebra of type $(2,2,1)$ has the quiver
$$
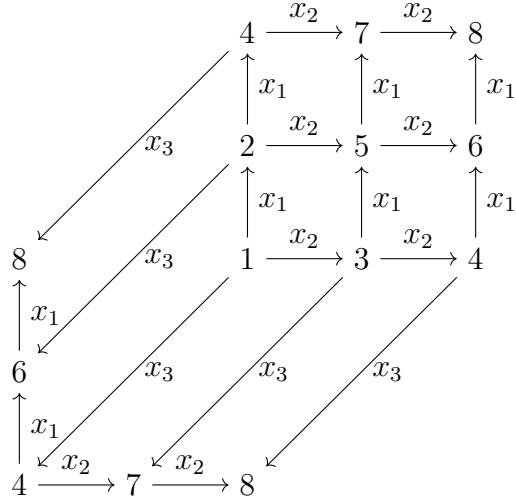
\begin{tikzpicture}[scale=1.5]
\node(a) at (0,0){$1$};
\node(b) at (0,1){$2$};
\node(c) at (0,2){$4$};
\node(d) at (1,0){$3$};
\node(e) at (1,1){$5$};
\node(f) at (1,2){$7$};
\node(g) at (2,0){$4$};
\node(h) at (2,1){$6$};
\node(i) at (2,2){$8$};
\node(j) at (-2,-2){$4$};
\node(k) at (-2,-1){$6$};
\node(l) at (-2,0){$8$};
\node(m) at (-1,-2){$7$};
\node(n) at (0,-2){$8$};

\path[->] (a) edgenode[right]{$x_1$} (b);
\path[->] (b) edgenode[right]{$x_1$} (c);
\path[->] (d) edgenode[right]{$x_1$} (e);
\path[->] (e) edgenode[right]{$x_1$} (f);
\path[->] (g) edgenode[right]{$x_1$} (h);
\path[->] (h) edgenode[right]{$x_1$} (i);
\path[->] (j) edgenode[right]{$x_1$} (k);
\path[->] (k) edgenode[right]{$x_1$} (l);

\path[->] (a) edgenode[above]{$x_2$} (d);
\path[->] (b) edgenode[above]{$x_2$} (e);
\path[->] (c) edgenode[above]{$x_2$} (f);
\path[->] (d) edgenode[above]{$x_2$} (g);
\path[->] (e) edgenode[above]{$x_2$} (h);
\path[->] (f) edgenode[above]{$x_2$} (i);
\path[->] (j) edgenode[above]{$x_2$} (m);
\path[->] (m) edgenode[above]{$x_2$} (n);

\path[->] (a) edgenode[right]{$x_3$} (j);
\path[->] (b) edgenode[right]{$x_3$} (k);
\path[->] (c) edgenode[right]{$x_3$} (l);
\path[->] (d) edgenode[right]{$x_3$} (m);
\path[->] (g) edgenode[right]{$x_3$} (n);
\end{tikzpicture}$$
where vertices with the same labels, and the corresponding arrows, should be identified. The relations are given by $x_1x_2=x_2x_1$, $x_2x_3=x_3x_2$ and $x_3x_1=x_1x_3$. Applying Proposition \ref{masod} we see that $f=e_1+e_2+e_4+e_6+e_8$ is a fabric idempotent together with $e=e_3+e_5+e_6+e_7+e_8$. The algebra $fAf$ is the $2$-canonical algebra of type $(2,1,1)$ with quiver 

$$
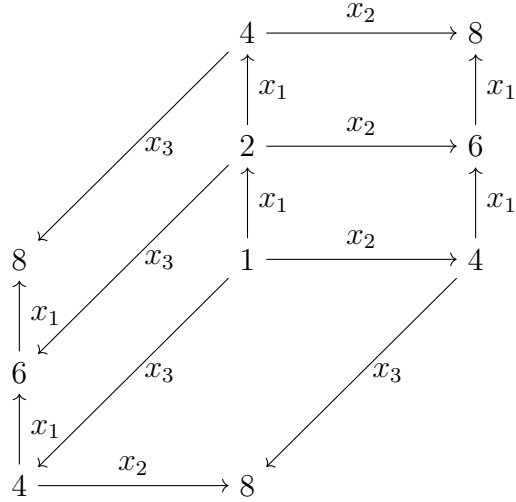
\begin{tikzpicture}[scale=1.5]
\node(a) at (0,0){$1$};
\node(b) at (0,1){$2$};
\node(c) at (0,2){$4$};
\node(g) at (2,0){$4$};
\node(h) at (2,1){$6$};
\node(i) at (2,2){$8$};
\node(j) at (-2,-2){$4$};
\node(k) at (-2,-1){$6$};
\node(l) at (-2,0){$8$};
\node(n) at (0,-2){$8$};

\path[->] (a) edgenode[right]{$x_1$} (b);
\path[->] (b) edgenode[right]{$x_1$} (c);
\path[->] (g) edgenode[right]{$x_1$} (h);
\path[->] (h) edgenode[right]{$x_1$} (i);
\path[->] (j) edgenode[right]{$x_1$} (k);
\path[->] (k) edgenode[right]{$x_1$} (l);

\path[->] (a) edgenode[above]{$x_2$} (g);
\path[->] (b) edgenode[above]{$x_2$} (h);
\path[->] (c) edgenode[above]{$x_2$} (i);
\path[->] (j) edgenode[above]{$x_2$} (n);

\path[->] (a) edgenode[right]{$x_3$} (j);
\path[->] (b) edgenode[right]{$x_3$} (k);
\path[->] (c) edgenode[right]{$x_3$} (l);
\path[->] (g) edgenode[right]{$x_3$} (n);
\end{tikzpicture}$$
The relations are again given by $x_1x_2=x_2x_1$, $x_2x_3=x_3x_2$ and $x_3x_1=x_1x_3$. We may apply Proposition \ref{masod} once more to show that the idempotent $f^\prime =e_1+e_4+e_8$ is a fabric idempotent of $fAf$. Finally $f^\prime fAff^\prime$ is the $2$-Beilinson algebra \cite{Be} with quiver

$$
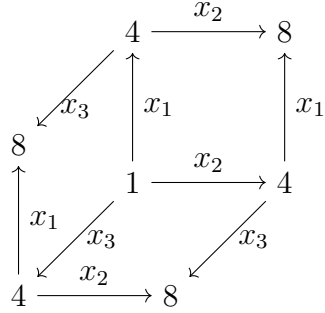
\begin{tikzpicture}[scale=1]
\node(a) at (0,0){$1$};
\node(c) at (0,2){$4$};
\node(g) at (2,0){$4$};
\node(i) at (2,2){$8$};
\node(j) at (-1.5,-1.5){$4$};
\node(l) at (-1.5,0.5){$8$};
\node(n) at (0.5,-1.5){$8$};

\path[->] (a) edgenode[right]{$x_1$} (c);
\path[->] (g) edgenode[right]{$x_1$} (i);
\path[->] (j) edgenode[right]{$x_1$} (l);

\path[->] (a) edgenode[above]{$x_2$} (g);
\path[->] (c) edgenode[above]{$x_2$} (i);
\path[->] (j) edgenode[above]{$x_2$} (n);

\path[->] (a) edgenode[right]{$x_3$} (j);
\path[->] (c) edgenode[below]{$x_3$} (l);
\path[->] (g) edgenode[right]{$x_3$} (n);
\end{tikzpicture}$$
and relations $x_1x_2=x_2x_1$, $x_2x_3=x_3x_2$ and $x_3x_1=x_1x_3$.
\end{eg}
 
\subsection{$n$-Auslander-Gorenstein algebras}
For an artin algebra $\Lambda$, define the \emph{dominant dimension} $\mathrm{dom.dim}(\Lambda)$ to be the number $n$ such that for a minimal injective resolution of $\Lambda$:
$$0\rightarrow \Lambda\rightarrow I_0\rightarrow \cdots \rightarrow I_{n-1}\rightarrow I_{n}\rightarrow \cdots$$ the modules $I_0,\ldots, I_{n-1}$ are projective-injective and $I_{n}$ is not projective. A subcategory $\mathcal{C}$ of $\mathrm{mod}(\Lambda)$ is \emph{precovering} if for any $M\in \mathrm{mod}(\Lambda)$ there is an object $C_M\in\mathcal{C}$ and a morphism $f:C_M\rightarrow M$ such that for any morphism $X\rightarrow M$ with $X\in \mathcal{C}$ factors through $f$; that there is a commutative diagram:

$$\begin{tikzcd} 
\ &X\arrow{d}\arrow[dotted]{dl}[above]{\exists}\\
C_M\arrow{r}{f}&M
\end{tikzcd}$$
The dual notion of precovering is \emph{preenveloping}.  A subcategory $\mathcal{C}$ that is both precovering and preenveloping is called \emph{functorially finite}.
An artin algebra $\Lambda$ is \emph{$n$-Auslander-Gorenstein algebra} \cite{is} if it satisfies $$\mathrm{dom.dim}(\Lambda)\geq n+1 \geq \mathrm{inj.dim}(_\Lambda\Lambda).$$
On the other hand, for a finite-dimensional algebra $A$, a functorially-finite subcategory $\mathcal{C}$ of $\mathrm{mod}(A)$ is an \emph{$n$-precluster tilting subcategory} if it satisfies the following conditions:
\begin{itemize}
\item The subcategory $\mathcal{C}$ is a generator-cogenerator for $\mathrm{mod}(A)$. 
\item There are inclusions $\tau_n(\mathcal{C})\subseteq \mathcal{C}$ and $\tau^-_n(\mathcal{C})\subseteq \mathcal{C}$.
\item There is an equality $\mathrm{Ext}^i_A(\mathcal{C},\mathcal{C})=0$ for $0<i<n$.
\end{itemize}
If $\mathcal{C}$ has an additive generator $M$, then $M$ is called an $n$-precluster-tilting module. Alternatively, $M$ is called an $n$-ortho-symmetric module in \cite{chenko}. 

\begin{prop}\cite[Proposition 4.3]{is}
Let $A$ be a finite-dimensional $k$-algebra, and $M$ an $n$-precluster tilting $A$-module with $n\geq 1$. Then $\Lambda=\mathrm{End}_A(M)^\mathrm{op}$ is an $n$-Auslander-Gorenstein algebra that satisfies $\mathrm{proj.inj}(\Lambda)=\mathrm{add}(_\Lambda A)$. 
\end{prop}
In fact, every $n$-Auslander-Gorenstein algebra arises in this fashion, but we will not use the full correspondence in this example. Let $Q$ be a finite quiver without loops or cycles and let $\overline{Q}$ be the \emph{double quiver} of $Q$, obtained from $Q$ be adding arrows $\overline{\alpha}:j\rightarrow i$ for $\alpha:i\rightarrow j\in Q$. The \emph{preprojective algebra} associated to $Q$ is the algebra $A=k\overline{Q}/I$ where $I$ is the ideal generated by $$I=\sum_{\alpha\in Q_1}\alpha\overline{\alpha}- \overline{\alpha}\alpha.$$ 
\begin{lm} \label{ausgor}
Let $Q$ be a quiver of type $A_n$ with vertices $\{1,2,\ldots,n\}$ and $A$ the associated preprojective algebra. Now let $P:=\bigoplus_{j \text{even}}P_j$ and $P^\prime:=\bigoplus_{j\text{odd}}P_j$ be projective-injective $A$-modules and $T:=\mathrm{rad}(P)$. 
Then $B:=\mathrm{End}_A(T\oplus A )^\mathrm{op}$ is $2$-Auslander-Gorenstein and has a fabric idempotent $f$, corresponding to the summand $T\oplus P^\prime$ (with respect to the idempotent $e$ corresponding to the summand $A$). %idemponents and summands.
\end{lm}
\begin{proof}
Let $i$ be an even vertex. Then there is an exact sequence 
\begin{align*}
0\rightarrow \mathrm{Hom}_A(T\oplus A,&\  \mathrm{rad}(P_i))\xrightarrow{f_i} \mathrm{Hom}_A(T\oplus A,P_i)\rightarrow  \mathrm{Hom}_A(T\oplus A,S_i)\rightarrow\\ &\rightarrow\mathrm{Ext}^1_A(T\oplus A,\mathrm{rad}(P_i))\rightarrow \mathrm{Ext}^1_A(T\oplus A,P_i)=0 .
\end{align*}
Then $\mathrm{coker}(f_i)$ is given by the homomorphisms from $T\oplus A$ to $P_i$ that do not factor through $ \mathrm{rad}(P_i)$. The only such homomorphism is the identity, $\mathrm{id}:P_i\rightarrow P_i$, so $\mathrm{dim}(\mathrm{coker}(f_i))=1$. Likewise the dimension of $ \mathrm{Hom}_A(T\oplus A,S_i)$ must be one, as by construction the only homomorphism from $T\oplus A$ to $S_i$ is the surjection $P_i\twoheadrightarrow S_i$. Therefore $ \mathrm{Hom}_A(T\oplus A,S_i)\cong \mathrm{coker}(f_i)\cong B/\langle 1-e_i\rangle$ where $e_i$ is the primitive idempotent of $B$ corresponding to the summand $P_i$ of $A\oplus T$. Consolidating over all even vertices results in the projective resolution
\begin{align}
0\rightarrow \mathrm{Hom}_A(T\oplus A, T)\rightarrow \mathrm{Hom}_A(T\oplus A,P)\rightarrow  _BB/\langle f \rangle\rightarrow 0
\end{align}
and the calculation \begin{align}\label{rigid}
\mathrm{Ext}^1_A(T,T)=0.
\end{align}
This shows that $\mathrm{proj.dim}_B(B/\langle f \rangle)\leq 1$. In addition, the conditions of Proposition \ref{masod} are easily seen to be satisfied, as any non-trivial $A$-homomorphism to $P$ factors through its radical. Hence, if $e$ is the idempotent corresponding to the summand $A$ of $T\oplus A$, then $\tau_B(B/\langle f\rangle)=DB/\langle e \rangle$. Note that by its definition, the idempotent $e$ determines a projective-injective $B$-module $Be$.

Next we show that $B$ is $2$-Auslander-Gorenstein, that is $\mathrm{add}(A\oplus T)$ is a 2-precluster tilting subcategory of $\mathrm{mod}(A)$. Certainly, the subcategory $\mathrm{add}(A\oplus T)$ is a generator-cogenerator of $\mathrm{mod}(A)$ and $\mathrm{Ext}^1_A(A\oplus T,A\oplus T)=0$ by  (\ref{rigid}). 
In fact $\tau_2(M)=M$  for any non-projective $A$-module $M$. This can be seen from the 2-Calabi-Yau properties of preprojective algebras. Specifically, by \cite[Proposition 4.10]{bbk}, which the authors attribute to Ringel and Schofield, $\Omega^3_{A^e}(A)\cong DA$, and hence $\Omega^3_A(M)\cong M\otimes_{A}\Omega^3_{A^e}(A)\cong \nu(M)$, using Lemma 2.2 of \cite{es}, and similarly $ \Omega^6(M)\cong M$. For a self-injective algebra, $\tau(M)\cong \Omega^2(\nu(M))$, so for any $A$-module $M$, we have $\tau(M)\cong \Omega^5(M)$ and therefore $\tau_2(M)=\tau(\Omega(M))\cong M$.  

Finally, to show that $f$ is a fabric idempotent, we may observe that the $B$-module $ \mathrm{Hom}_A(T\oplus A,P)$ is projective-injective and hence $ _BB/\langle f\rangle=\Omega_B^{-1}(B)$. Now the properties $B/\langle f\rangle=\Omega^2_B(DB)$ and   $DB\in \mathrm{gen}_1(Be)$ are equivalent to the $2$-Auslander-Gorenstein condition $\mathrm{domdim}(B)\geq 3 \geq \mathrm{inj.dim}(_BB).$
\end{proof}

\begin{remark}
Recall that for a finite-dimensional algebra $A$, an $A$-module $T$ is a \emph{tilting module} if 
\begin{itemize}
	\item $\mathrm{proj.dim}(T)\leq 1$,
	\item $\mathrm{Ext}^1_A(T,T)=0$ and
	\item There is a short exact sequence $$0\rightarrow _AA\rightarrow T_0\rightarrow T_1\rightarrow 0$$ with $T_0,\ T_1\in  \mathrm{add}(T)$.
\end{itemize}
A tilting module is \emph{$P$-special} \cite{ps} if $T_0\in \mathrm{add}(P)$ for some projective $A$-module $P$. 
If $A$ is a finite-dimensional algebra with fabric idempotent $f$, then it follows from the definition of a fabric idempotent as well as Proposition \ref{masod} that there is an idempotent $e$ of $A$ and an exact sequence $$0\rightarrow A(1-e)\rightarrow P \rightarrow A/\langle f\rangle\rightarrow 0,$$ where $P\in \mathrm{add}(Ae)$.  It follows that $\mathrm{Ext}_A^1(A/\langle f\rangle,A/\langle f\rangle)=0$ from the fact that $\mathrm{mod}(A/\langle f\rangle)$ is closed under extensions in $\mathrm{mod}(A)$. Furthermore, condition (\ref{ef1}) of Proposition \ref{masod} ensures that $\mathrm{Ext}^1_A(A/\langle f\rangle, Ae)=0$. Therefore the $A$-module $Ae\oplus A/\langle f\rangle$ is an $Ae$-special tilting module. 

For an algebra $A$ such that $\mathrm{dom.dim}(A)\geq 1$ and with projective-injective generator $Ae$, the module $Ae\oplus \Omega^{-1}(A)$ is a \emph{canonical tilting module} in the sense of \cite[Section 4]{cx}. Canonical tilting modules also appear in the articles \cite{cbs} and \cite{nrtz}. 
In the case of $n$-Auslander-Gorenstein algebras, canonical tilting modules have been studied in \cite{mar2} and \cite{ps}. 
\end{remark}

%\begin{theorem}[Auslander-Solberg]\cite[Theorem 2.7]{is}
%For an ideal $\mathcal{I}$, let $$F:=F^\mathcal{I}=F^{\mathrm{mod}(A/\mathcal{I})\cup\mathrm{add}(A)}=F^\mathcal{X},$$  $$G:=F^{\tau\mathcal{I}}=F_\mathcal{I}=F_{\mathrm{mod}(A/\mathcal{I})\cup\mathrm{add}(DA)}=F_\mathcal{X}.$$ Then $\tau$ induces a functorial isomorphism $$\mathrm{Ext}_F^i(M,N)\cong \mathrm{Ext}_G^i(\tau(M),\tau(N))$$ for all $M,\ N\in\mathrm{mod}(A)$ and $i\geq 1$. 
%\end{theorem}

\begin{eg}
Let $A=kQ/I$ be the algebra given by the quiver $Q$:
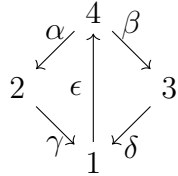
$$\begin{tikzpicture}
\node(d) at (0,1){$4$};
\node(a) at (0,-1){$1$};
\node(c) at (1,0){$3$};
\node(b) at (-1,0){$2$};

\path[->] (d) edgenode[above]{$\alpha$} (b);
\path[->] (d) edgenode[above]{$\beta$} (c);
\path[->] (b) edgenode[below]{$\gamma$} (a);
\path[->] (c) edgenode[below]{$\delta$} (a);
\path[->] (a) edgenode[left]{$\epsilon$} (d);
\end{tikzpicture}$$
with relations defined by $$I=<\gamma\alpha-\delta\beta,\epsilon\gamma\alpha,\alpha\epsilon\gamma,\beta\epsilon\delta>.$$
Then $P_1=I_1$, $P_2=I_3$, $P_3=I_2$ are projective-injective $A$-modules and the exact sequence:

$$0\rightarrow P_4\rightarrow P_1\rightarrow P_1\rightarrow P_2\oplus P_3\rightarrow I_4\rightarrow 0$$ shows that $A$ is $2$-Auslander-Gorenstein, and the exact sequence $$0\rightarrow P_4\rightarrow P_1\rightarrow S_1\rightarrow 0$$ shows that $f=e_2+ e_3+ e_4$ is a fabric idempotent. In fact, setting $e=e_1+ e_2+ e_3$, then the algebra $eAe$ is a preprojective algebra with quiver
$$\begin{tikzpicture}[scale=2]
\node(d) at (0,0){$4$};
\node(c) at (1,0){$3$};
\node(b) at (-1,0){$2$};

\node(da) at (0.1,0.1){};
\node(db) at (-0.1,0.1){};
\node(dc) at (0.1,-0.1){};
\node(dd) at (-0.1,-0.1){};

\node(ca) at (0.9,0.1){};
\node(cb) at (0.9,-0.1){};

\node(ba) at (-0.9,0.1){};
\node(bb) at (-0.9,-0.1){};

\path[->] (da) edgenode[above]{$\alpha$} (ca);
\path[->] (db) edgenode[above]{$\beta$} (ba);
\path[->] (bb) edgenode[below]{$\gamma^\prime$} (dd);
\path[->] (cb) edgenode[below]{$\delta^\prime$} (dc);
\end{tikzpicture}$$
and hence $A$ is a $2$-Auslander-Gorenstein algebra as described in Lemma \ref{ausgor}.\end{eg}

\subsection{Higher Nakayama algebras}

One of the motivating examples for studying localising subcategories are Nakayama algebras \cite{cy}. A \emph{Nakayama algebra} is defined to be an algebra where each projective module is uniserial. It is known that the quiver of any finite-dimensional Nakayama algebra is either of type $A$ or type $\tilde{A}$, and is completely determined by a \emph{Kupisch series}, see the references in \cite{jk}.  
\begin{defn}
Let $\underline{l}=(l_0,l_1,\ldots,l_{k-1})$ be an $k$-tuple of positive integers. Then $\underline{l}$ is a finite Kupisch series if 
\begin{itemize}
 \item If $1\leq i\leq k-1$, then $l_i\geq 2$.
\item If $0\leq i< k-1$, then $l_i-l_{i+1}\leq 1$.
\item If $l_0\ne 1$, then $l_{k-1}-l_0\leq 1$. 
\end{itemize}
\end{defn}

There is a higher version of the Nakayama algebras, as introduced in \cite{jk}. 
Let $l\geq 2$ be an integer and $\mathbf{S}^{(n)}_l$ the set of $n$-subsets of $\mathbb{Z}$ such that for each subset $I=\{i_1,\ldots,i_n\}\in\mathbf{S}^{(n)}_l$, we have $i_1<i_2<\vdots <i_n<i_1 +n+l-1$. Let $Q^{(n)}_{\mathbb{Z}l}$ be the quiver with vertices indexed by the elements of $\mathbf{S}^{(n)}_l$ and arrows $\alpha_{i}(I):I\rightarrow J$ wherever $I\setminus\{i\}=J\setminus\{i+1\}$ for some $i\in I$. For a $n$-subset  $I=\{i_1,i_2,\ldots,i_n\}$ set $$\tau_d(I)=\{i_1+1,\ldots,i_n+1\}. $$ Let $I^{(n)}_l$ be the admissible ideal of $Q^{(n)}_{\mathbb{Z}l}$ generated by the elements $$\alpha_j(\alpha_i(I))-\alpha_i(\alpha_j(I)),$$ which range over the elements of $\mathbf{S}^{(n)}_l$. By convention, $\alpha_i(I)=0$ whenever $I$ or $I\cup\{i+1\}\setminus \{i\}$ is not a member of $\mathbf{S}^{(n)}_l$, hence there are zero relations included in the ideal $I^{(n)}_l$. Now define
\begin{align*}A^{(n)}_{\mathbb{Z}l}&:=\mathbf{k}Q^{(n)}_{\mathbb{Z}l}/I^{(n)}_l.\\
A^{(n)}_{l}&:=A_{\mathbb{Z}l}/\tau^k_n.\end{align*}

\begin{theorem}\cite[Theorem 4.10]{jk}\label{id}
Let $l,n$ be positive integers and $l\geq 2$. Then the algebra $A^{(n)}_{l}$ is self-injective. 
\end{theorem}

Consider a Kupisch series $\underline{l}=(l_0,l_1,\ldots,l_{k-1})$ with $l_0\ne 1$. Set $$l:=\mathrm{max}\{l_i\ |0\leq i\leq k-1\}$$ and $l_{j+ik}:=l_j$ for $i\in \mathbb{Z}$. Let 
$$ H^{(n)}_{\mathbb{Z}\underline{l}}:= \{\{i_1,\ldots,i_n\}\in\mathbf{S}^{(n)}_l|i_j\geq j+l_{j}+k-1\}.$$
Let $e^{(n)}_{\underline{l}}$ be the idempotent of $A^{(n)}_{l}$ corresponding to the vertex set $H^{(n)}_{\mathbb{Z}\underline{l}}$. Then define the algebras
\begin{align*} A^{(n)}_{\mathbb{Z}\underline{l}}&:=A^{(n)}_{l}/\langle e^{(n)}_{\underline{l}}\rangle .\\ 
A^{(n)}_{\underline{l}}&:=A_{\mathbb{Z}\underline{l}}/\tau^k_n\end{align*}
The algebra $A^{(n)}_{\underline{l}}$ is a \emph{higher Nakayama algebra of type $\tilde{A}$}. Note that for a \emph{constant Kupisch series} $\underline{l}=(l,l,\ldots,l)$ and $l\geq 2$, then $A^{(n)}_{\underline{l}}=A^{(n)}_{l}$. 

For a  higher Nakyama algebra with Kupisch series $\underline{l}=(l_0,l_1,\ldots,l_{k-1})$ such that $l_0\ne 1$ and $l_{k-1}\leq l_0$, $l_1<l_0$, define the new Kupisch series $\underline{l}^\prime$ via the following procedure. Let $0\leq i<k$ be given, and let $I=\{i=i_1,i_2,\ldots,i_n\}$ be the minimal subset (in that the difference $i_j-i+j$ is coordinate-wise minimal) such that for all $2\leq j\leq k$ we have $i_j\not\equiv 0$ ($\mathrm{mod}\ k$). If no such subset exists, or $i\equiv 0$ ($\mathrm{mod}\ k)$, then define $l_i^\prime =0$. Else, let $$l_i^\prime:=\#\{i_n\leq m< i+l_i+n-1|m\not\equiv 0(\mathrm{mod}\ k) \}.$$
Then let $\underline{l}^\prime$ be the list of each $l_i^\prime$ such that $l_i^\prime \geq 1$ (under the same ordering as for $\underline{l}$). This is a Kupisch series if each $l_i^\prime\geq 2$: by definition $$l_{k-1}-l_{1}= (l_{k-1}-l_0)+(l_0-l_1)\leq 2,$$ but we have assumed that $l_{k-1}-l_0\leq 0$ and the definition of $\underline{l}^\prime$ ensures that also $l^\prime_{k-1}-l^\prime_1 \leq 1$.

\begin{eg}\label{cup}
Let $\underline{l}=(4,3,3,3)$ and $n=2$. Then for each $1\leq i\leq 4$, the minimal $2$-subsets $\{i,i_2\}$ such that $i_2\not\equiv 0$ ($\mathrm{mod}\ 4$) are $\{1,2\}$, $\{2,3\}$, $\{3,5\}$. We calculate $$l_1=\#\{2\leq m\leq 4|m\not\equiv 0(\mathrm{mod}\ k) \}=2.$$ In this fashion we obtain the Kupisch series $\underline{l}^\prime=(2,2,2)$. 
\end{eg}
 The following theorem is a generalisation of \cite[Corollary 3.1]{cy}.
\begin{theorem}\label{nakay}
Let $A$ be a (finite-dimensional) higher Nakayama algebra. Then there exists a (finite-dimensional) self-injective higher Nakayama algebra $B$, such that 
$$D_{\mathrm{sg}}(A)\cong D_{\mathrm{sg}}(B)\cong \underline{\mathrm{mod}}(B).$$ In addition, we may list idempotents $1_A=f_0,f_1,f_2,\ldots,f_m$ such that $f_i$ is a fabric idempotent of $f_{i-1}\ldots f_0 A f_0\ldots f_{i-1}$ and $$B=f_{m}\ldots f_2f_1 A f_1f_2\ldots f_{m}.$$
\end{theorem}

\begin{proof}
If $\underline{l}$ is a Kupisch series where $l_0=1$, then $A:=A^{(n)}_{\underline{l}}$ has finite global dimension, and $D_{\mathrm{sg}}(A)$ is trivial. So let $\underline{l}$ be a Kupisch series of length $k$ and $l_0\ne 1$. If $\underline{l}$ is a constant Kupisch series, then $A^{(n)}_{\underline{l}}$ is a self-injective algebra by Theorem \ref{id}. In this case there is nothing to prove. So assume without loss of generality that $l_{0}>l_1$ and $l_{0}\geq l_{k-1}$.

Let $A:=A^{(n)}_{\underline{l}}$ be a higher Nakayama algebra satisfying the assumptions above. We show that we may successively apply Corollary \ref{chenii} to obtain the higher Nakayama algebra $A^{(n)}_{\underline{l}^\prime}$ whenever $l^\prime_i\geq 2$ for all $1\leq i\leq k-1$, and an algebra of finite global dimension otherwise.
  
Let $x:=\{0,i_2,\ldots,i_n\}$ be a vertex in the quiver of $A^{(n)}_{\underline{l}}$ such that $i_2>1$. In addition, let $x^\prime:=\{1,i_2,\ldots,i_n\}$; then any non-zero path between $x$ and any vertex $\{j_1,j_2,\ldots, j_n\}$ where $j_1>0$, must factor through the vertex $x^\prime$. As $l_1<l_0$, there is an inclusion in $\mathrm{mod}(A^{(n)}_{\mathbb{Z}\underline{l}})$: $$P_{x^\prime}\hookrightarrow P_x.$$
Now, define the following idempotent of $A$, $$f_1:=\sum_{i_1\ne 0}e_{i_1,\ldots ,i_n},$$ 
and the idempotent of $A^{(n)}_{\mathbb{Z}\underline{l}}$: $$f_1^\prime:=\sum_{i_1\not\equiv 0(\mathrm{mod}\ k)}e_{i_1,\ldots, i_n}.$$
 This induces a short exact sequence 
$$0\rightarrow Ae_{x^\prime}\rightarrow Ae_x\rightarrow Ae_x/\langle f_1 \rangle\rightarrow 0.$$
In addition, observe that if $i=(0,1,i_3,\ldots,i_n)$, then $Ae_i/\langle f_1\rangle$ is also projective as an $A$-module. So $\mathrm{proj.dim}_A(A/\langle f_1 \rangle )\leq 1$. 
By construction, the algebra $A/\langle f_1 \rangle$ is an $(n-1)$-Auslander algebra, and $\mathrm{gl.dim}(A/\langle f_1\rangle)=n-1$. Therefore the conditions of Corollary \ref{chenii} are satisfied, and there is an equivalence $D_{\mathrm{sg}}(A)\cong D_\mathrm{sg}(f_1Af_1)$.

Let $B:=f_1Af_1$ and $B^\prime :=f_1^\prime A^{(n)}_{\mathbb{Z}\underline{l}}f_1^\prime$. For any vertex $x:=(i_1,i_2,\ldots,i_n)$ in the quiver of $B^\prime$ such that $i_2\equiv 0\ (\mathrm{mod}\ k)$ and $i_3>i_2+1$, let $x^\prime:=(i_1,i_2+1,\ldots,i_n)$. Then any non-zero path between $x$ and any vertex $(j_1,j_2,\ldots, j_n)$ where $j_2>j_1$, must factor through the vertex $x^\prime$.  Let $\underline{j}^\prime:=(j^\prime_1,j^\prime_2,\ldots, j^\prime_n)$ be a vertex in the quiver of $A^{(n)}_{\mathbb{Z}\underline{l}}$ such that $j^\prime_2>i_2$. Suppose there is no non-zero path from $x$ to $\underline{j}^\prime$ but a non-zero path from $x^\prime$ to $\underline{j}^\prime$. Then we must have that $j^\prime_1\equiv 0\ (\mathrm{mod}\ k)$. The consequence is that there is an inclusion in $\mathrm{mod}(B^\prime)$: $$P_{x^\prime}\hookrightarrow P_x.$$ 
So define the following idempotent of $B$, $$f_2:=\sum_{i_2\not\equiv 0(\mathrm{mod}\ k)}e_{i_1,\ldots ,i_n}.$$ Then arguing as before (the only change is that $B/\langle f_2\rangle$ may instead be a factor algebra of an $(n-1)$-Auslander algebra of type $A$) we obtain an equivalence $$D_{\mathrm{sg}}(B)\cong D_\mathrm{sg}(f_2Bf_2).$$ Continuing in this fashion, we obtain an equivalence $$D_{\mathrm{sg}}(A)\cong D_\mathrm{sg}(fAf)$$ where $$f=\sum_{\substack{ i_j\not\equiv 0(\mathrm{mod}\ k) \\ \forall 1\leq j\leq n}}e_{i_1,\ldots ,i_n}.$$

It now follows directly from the construction that if $l_i^\prime\geq 2$ for all $1\leq i\leq k-1$, then $fAf$ is the higher Nakayama algebra $A^{(n)}_{\underline{l}^\prime}$. If $l_i<2$ for some $1\leq i\leq k-1$, then the algebra $fAf$ has no cycles, and hence $D_\mathrm{sg}(fAf)$ is trivial. If $A^{(n)}{\underline{l}^\prime}$ is not self-injective, then we may apply the construction again until we reach a self-injective higher Nakayama algebra.  Finally, it follows from the definition of a self-injective algebra that if $B$ is a self-injective algebra, then $D_{\mathrm{sg}}(B)\cong \underline{\mathrm{mod}}(B)$.
\end{proof}

\begin{eg}

Let $\underline{l}= (4,3,3,3)$, then the higher Nakayama algebra $A:=A^2_{\underline{l}}$ has quiver

$$
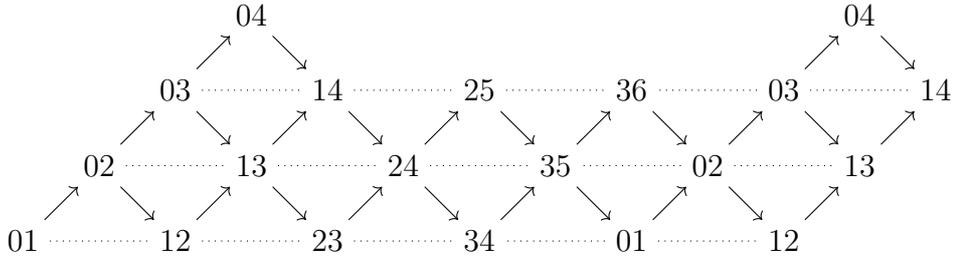
\begin{tikzpicture}
\node(a) at (0,0){$01$};
\node(b) at (1,1){$02$};
\node(c) at (2,2){$03$};
\node(d) at (3,3){$04$};
\node(e) at (2,0){$12$};
\node(f) at (3,1){$13$};
\node(g) at (4,2){$14$};
\node(h) at (4,0){$23$};
\node(i) at (5,1){$24$};
\node(j) at (6,2){$25$};
\node(k) at (6,0){$34$};
\node(l) at (7,1){$35$};
\node(m) at (8,2){$36$};
\node(n) at (8,0){$01$};
\node(o) at (9,1){$02$};
\node(p) at (10,2){$03$};
\node(q) at (11,3){$04$};
\node(r) at (10,0){$12$};
\node(s) at (11,1){$13$};
\node(t) at (12,2){$14$};

\path[->] (a) edge (b);
\path[->] (b) edge (c);
\path[->] (c) edge (d);
\path[->] (e) edge (f);
\path[->] (f) edge (g);
\path[->] (h) edge (i);
\path[->] (i) edge (j);
\path[->] (k) edge (l);
\path[->] (l) edge (m);
\path[->] (n) edge (o);
\path[->] (o) edge (p);
\path[->] (p) edge (q);
\path[->] (r) edge (s);
\path[->] (s) edge (t);

\path[->] (b) edge (e);
\path[->] (c) edge (f);
\path[->] (d) edge (g);
\path[->] (f) edge (h);
\path[->] (g) edge (i);
\path[->] (i) edge (k);
\path[->] (j) edge (l);
\path[->] (l) edge (n);
\path[->] (m) edge (o);
\path[->] (o) edge (r);
\path[->] (p) edge (s);
\path[->] (q) edge (t);

\path[-,dotted] (a) edge (e);
\path[-,dotted] (e) edge (h);
\path[-,dotted] (h) edge (k);
\path[-,dotted] (n) edge (k);
\path[-,dotted] (n) edge (r);

\path[-,dotted] (b) edge (f);
\path[-,dotted] (f) edge (i);
\path[-,dotted] (i) edge (l);
\path[-,dotted] (o) edge (l);
\path[-,dotted] (o) edge (s);

\path[-,dotted] (c) edge (g);
\path[-,dotted] (g) edge (j);
\path[-,dotted] (j) edge (m);
\path[-,dotted] (p) edge (m);
\path[-,dotted] (p) edge (t);
\end{tikzpicture}$$
and relations indictated by the dotted arrows. It may be calculated that $A$ is 6-Iwanaga-Gorenstein, and setting 
\begin{align*}
f&:=e_{12}+ e_{13}+ e_{14} + e_{23}+ e_{24}+ e_{25}+ e_{34}+ e_{35}+ e_{36}\\
e&:=e_{01}+ e_{02}+ e_{03}+ e_{04} + e_{23}+ e_{24}+ e_{25}+ e_{34}+ e_{35}+ e_{36}
\end{align*}
then it may be seen by Proposition \ref{masod} that $f$ is a fabric idempotent and that $\Omega^2(I_{03})=P_{14}$,  $\Omega^6(I_{13})= P_{12}$ and $\Omega^6(I_{23})=P_{13}$. This implies that for any module $M\in\mathrm{GI}(A)$, the first two terms in the projective resolution of $M$ are generated by $Ae$, the first six generated by $A(e+ e_{14})$ and the rest generated by $Af$.
In addition, $D_{\mathrm{sg}}(A)\cong D_{\mathrm{sg}}(fAf)$,
where $fAf$ is the algebra given by the quiver with mesh relations indicated
$$
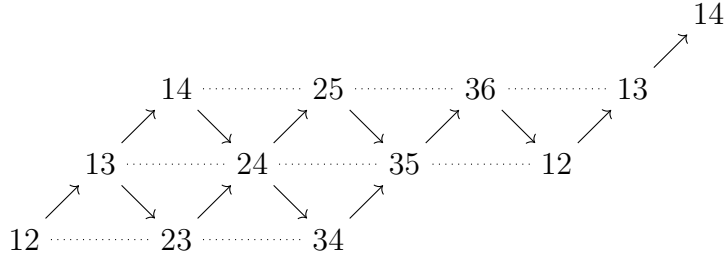
\begin{tikzpicture}
\node(e) at (2,0){$12$};
\node(f) at (3,1){$13$};
\node(g) at (4,2){$14$};
\node(h) at (4,0){$23$};
\node(i) at (5,1){$24$};
\node(j) at (6,2){$25$};
\node(k) at (6,0){$34$};
\node(l) at (7,1){$35$};
\node(m) at (8,2){$36$};
\node(r) at (9,1){$12$};
\node(s) at (10,2){$13$};
\node(t) at (11,3){$14$};

\path[->] (e) edge (f);
\path[->] (f) edge (g);
\path[->] (h) edge (i);
\path[->] (i) edge (j);
\path[->] (k) edge (l);
\path[->] (l) edge (m);
\path[->] (r) edge (s);
\path[->] (s) edge (t);

\path[->] (f) edge (h);
\path[->] (g) edge (i);
\path[->] (i) edge (k);
\path[->] (j) edge (l);

\path[-,dotted] (e) edge (h);
\path[-,dotted] (h) edge (k);

\path[-,dotted] (f) edge (i);
\path[-,dotted] (i) edge (l);
\path[-,dotted] (r) edge (l);

\path[-,dotted] (g) edge (j);
\path[-,dotted] (j) edge (m);
\path[-,dotted] (s) edge (m);

\path[->] (m) edge (r);
\end{tikzpicture}$$
Setting 
\begin{align*}
f^\prime&:=e_{12}+ e_{13} + e_{23}+ e_{25}+ e_{35}+ e_{36}\\
e^\prime&:=e_{12}+ e_{13}+ e_{14} + e_{23}+ e_{25}+ e_{36},
\end{align*}
then as in the proof of Theorem \ref{nakay} $f^\prime$ is a fabric idempotent of $fAf$. Finally the algebra $f^\prime fAf f^\prime$ is self-injective:

$$
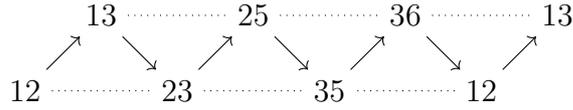
\begin{tikzpicture}
\node(e) at (2,0){$12$};
\node(f) at (3,1){$13$};
\node(h) at (4,0){$23$};
\node(j) at (5,1){$25$};
\node(l) at (6,0){$35$};
\node(m) at (7,1){$36$};
\node(r) at (8,0){$12$};
\node(s) at (9,1){$13$};

\path[->] (e) edge (f);
\path[->] (f) edge (h);
\path[->] (h) edge (j);
\path[->] (j) edge (l);
\path[->] (l) edge (m);
\path[->] (m) edge (r);
\path[->] (r) edge (s);

\path[-,dotted] (e) edge (h);
\path[-,dotted] (h) edge (l);
\path[-,dotted] (r) edge (l);

\path[-,dotted] (f) edge (j);
\path[-,dotted] (j) edge (m);
\path[-,dotted] (m) edge (s);
\end{tikzpicture}$$

and so we may calculate $D_{\mathrm{sg}}(A)\cong D_{\mathrm{sg}}(f^\prime fAff^\prime)\cong \underline{\mathrm{mod}}(f^\prime fAff^\prime)$

\end{eg}

\section{Acknowledgements}
This paper was completed as part of my PhD studies, with the support of the Austrian Science Fund (FWF): W1230. I would like to thank my supervisor, Karin Baur, for her continued help and support during my studies.

\bibliographystyle{amsplain}
\bibliography{idem}

\end{document}